%% file: samplepaper.tex
\documentclass[runningheads]{llncs}
\usepackage[T1]{fontenc}

\input{math_commands.tex}

\usepackage{graphicx}

\usepackage{hyperref}
\usepackage{url}
\usepackage{algorithm}
\usepackage{algpseudocode}
\usepackage{amsmath}
\usepackage{amsthm}
\usepackage{amsthm}

\usepackage{amsthm}

\newtheoremstyle{dotstyle}
  {\topsep}   
  {\topsep}   
  {\itshape}  
  {}          
  {\bfseries} 
  {.}         
  {0.5em}     
  {}          

\theoremstyle{dotstyle}

\newtheorem{assumption}{Assumption}   

\newcommand{\gk}{\tilde g(x_k)}
\newcommand{\fk}{f(x_k)}
\newcommand{\nfk}{\nabla f(x_k)}
\newcommand{\nnfk}{\|\nabla f(x_k)\|}
\newcommand{\fnfk}{\dfrac{\nfk}{\nnfk}}
\newcommand{\Rk}{\|x_k-x^*\|}
\newcommand{\Rkn}{\|x_{k+1}-x^*\|}
\newcommand{\inner}[2]{\left\langle #1, #2 \right\rangle}
\newtheoremstyle{dotstyle_def}
  {\topsep}{\topsep}{\normalfont}{}{\bfseries}{.}{0.5em}{}

\theoremstyle{dotstyle_def}

\newenvironment{e*}
  {\begin{equation*}\begin{aligned}}
  {\end{aligned}\end{equation*}}

\begin{document}
\title{Normalized First-Order Methods for (Quasi)-Convex \((L_0,L_1)\)-Smooth Optimization with Inexact Gradients}
%
%

\author{Evgeniy Kovalev \inst{1}\orcidID{0009-0001-6024-4563} \and
Fedor Stonyakin\inst{1,2}\orcidID{0000-0002-9250-4438}}
\authorrunning{F. Author et al.}
%
\institute{Innopolis University, Russia \and
Moscow Institute of Physics and Technology, Russia}
\maketitle              
\begin{abstract}
Generalized smoothness, such as \((L_0,L_1)\)-smoothness, have recently attracted considerable attention due to their ability to model optimization problems arising in modern machine and deep learning, where the classical Lipschitz assumptions of the gradient is often violated. At the same time, computing exact gradients may be impractical or computationally expensive in many applications. In this work, we study convex \((L_0,L_1)\)-smooth optimization under access only to a normalized approximation recently proposed Comparison Oracle, which returns an inexact normalized gradient in linear time with a bounded absolute error. Within this framework, we develop comparison-oracle variants of Normalized Gradient Descent (NGD) and Gradient Descent with Polyak stepsizes. In setting of NGD, we rely on strictly weaker assumption of quasi-convexity rather than convexity. We establish explicit upper bounds on the approximation error that guarantee convergence and derive convergence rates for all proposed methods. Unlike existing analyses, our results require neither classical smoothness assumptions nor access to exact gradients or their exact normalized counterparts. Finally, numerical experiments corroborate the theoretical findings. 

\keywords{Generalized Smoothness \and Comparison Oracle \and Inexact Gradient.}
\end{abstract}

\section{Introduction}\label{sec:introduction}
\subsection{Related Works}
Many optimization problems arising in modern machine learning violate the classical assumption of Lipschitz continuity of the gradient. This has motivated the development of generalized smoothness conditions that extend the applicability of first-order optimization methods beyond the standard smooth setting. Among them, the notion of \((L_0,L_1)\)-smoothness, introduced by \cite{zhang2020gradient}, has received considerable attention due to its ability to capture the behavior of practical models, including modern neural network architectures. Another reason why this weaker assumption is interesting for Deep Learning problems is because for many functions the constant $L$ can be much larger than constants $L_0$ and $L_1$ (see examples in \cite{gorbunov2024methodsl0l1}). Subsequently, several works extended this framework and established convergence guarantees for a variety of first-order methods under generalized smoothness assumptions.
\cite{zhang2020improved,chen2023generalizedsmooth,gorbunov2024methodsl0l1,takezawa2024parameterfree,koloskova2023revisiting,vyguzov2025frankwolfe}. In particular, \cite{chen2023generalizedsmooth} introduced the broad class of \(\alpha\)-symmetric \((L_0,L_1)\)-smooth functions, unifying previous definitions and showing that important problems such as phase retrieval and distributionally robust optimization satisfy this assumption. Throughout this paper, we adopt this generalized definition of \((L_0,L_1)\)-smoothness.

Existing analyses of optimization methods under \((L_0,L_1)\)-smoothness assume access to exact first-order information or its deterministic variants. However, in many applications computing the exact gradient is either computationally expensive or impossible. This issue arises, for example, in large-scale machine learning, distributed optimization, and gradient-free settings \cite{nesterov2017gradientfree,liu2020primer,zhouFL}. As a result, there is growing interest in developing optimization algorithms that remain reliable under inexact gradient information. Despite the extensive literature on generalized smoothness, the convergence behavior of fundamental gradient methods under inexact gradient information has remained largely unexplored.

In this work, we study convex \((L_0,L_1)\)-smooth optimization in the setting where only an approximation of the normalized gradient is available. This type of inexactness is  caused by using, for example, recently developed Comparison Oracle \cite{comparisons}. Building upon this assumption, we develop variants of Normalized Gradient Descent and Gradient Descent with Polyak stepsizes. For both methods, we establish explicit upper bounds on the admissible approximation error that guarantee convergence and derive the corresponding convergence rates. Notably, we show that Normalized Gradient Descent attains the same convergence rate under the weaker assumption of quasi-convexity as it does under convexity. Moreover, we suggest an adaptive version of the Polyak method, which does not require prior knowledge of the upper bound of the gradient norm along the trajectory. To the best of our knowledge, these are the first convergence guarantees for gradient methods under \((L_0,L_1)\)-smoothness that rely solely on approximate gradient information. Finally, numerical experiments corroborate the theoretical analysis and illustrate the practical behavior of the proposed algorithms.

\subsection{Our Contribution}

Building upon the comparison oracle framework, we propose adaptations of Normalized Gradient Descent (NGD) and the Gradient Descent with Polyak step sizes (Polyak-GD) that operate under inexact gradient information. The main contributions of this work are summarized as follows:

\begin{enumerate}
    \item Relying exclusively on approximate, normalized gradient directions subject to an absolute error $\delta$, we develop variants of Normalized Gradient Descent (NGD) for $(L_0,L_1)$-smooth quasi-convex functions and Polyak-GD for $(L_0,L_1)$-smooth convex functions.
    \item For both algorithms, we establish convergence rate and explicit upper bounds on the admissible approximation error $\delta$ required to guarantee convergence.
    \item We propose an adaptive formulation of Polyak-GD that eliminates the requirement of knowing an upper bound on the gradient norm in advance, demonstrating superior performance over standard baselines in numerical experiments.
    \item We perform numerical evaluations on synthetic $(L_0, L_1)$-smooth functions to validate our theoretical error bounds and illustrate the practical efficiency of both methods.
\end{enumerate}

To the best of our knowledge, these results provide the first formal convergence guarantees for gradient-based methods under generalized $(L_0, L_1)$-smoothness using solely approximate gradient evaluations.

\section{Problem Statement and Assumptions}\label{sec:assumptions}

In this work, we consider the following unconstrained minimization problem
\begin{equation}\label{eq:problem}
    \min_{x \in \mathbb{R}^d} f(x),
\end{equation}
where $f : \mathbb{R}^d \to \mathbb{R}$ is (strongly) convex differentiable function.
\begin{assumption}[Convexity]
\label{assum:convexity}
Function $f:\mathbb{R}^{d}\rightarrow\mathbb{R}$ is convex ($\mu$-strongly convex) with $\mu = 0$ ($\mu > 0)$. That is, for all $x, y \in \mathbb{R}^d$ we have
\begin{align}
    f(y) \ge f(x) + \langle\nabla f(x), y-x\rangle + \frac{\mu}{2}\|x-y\|^2.
\end{align}
\end{assumption}
The following assumption is used only in Section~\ref{ngd}, where it replaces Assumption~\ref{assum:convexity}.
\begin{assumption}[Quasi-convexity]
\label{assum:quasi-convexity}
Function $f:\mathbb{R}^{d}\rightarrow\mathbb{R}$ is differentiable and quasi-convex. That is, for all $x, y \in \mathbb{R}^d$ satisfying
\begin{align*}
        f(y)\le f(x),
\end{align*}
it holds that
\begin{align*}
    \langle \nabla f(x),\, y-x\rangle \le 0.
\end{align*}
\end{assumption}

As we mentioned, the objective function is considered to be $(L_0,L_1)$-smooth.

\begin{assumption}[$(L_0,L_1)$-smoothness]
\label{assum:smoothness}
Function $f:\mathbb{R}^{d}\rightarrow\mathbb{R}$ is called $(L_0,L_1)$-smooth if for all $x, y \in \mathbb{R}^d$ we have
\begin{align}
    \|\nabla f(x)-\nabla f(y)\| \le \left( L_0 + L_1 \sup_{z\in[x,y]}\|\nabla f(z)\| \right) \|x-y\|.
\end{align}
\end{assumption}

The following lemma, established in \cite{chen2023generalizedsmooth}, will serve as a useful tool throughout the subsequent sections of this paper.

\begin{lemma}\label{lemma:techlemma1}
    Let $f$ be $(L_0,L_1)$-smooth function. Then, for any  $x, y \in \mathbb{R}^d$ we have:
    \begin{equation*}
        f(y) \le f(x) + \langle \nabla f(x), y - x \rangle + \frac{L_0 + L_1 \|\nabla f(x)\|}{2} \exp(L_1 \|x - y\|) \|x - y\|^2.
    \end{equation*}
\end{lemma}

Our next lemma plays a central role in the presented theoretical proofs, whose tighter estimates allow us to establish sharp convergence rates and bounds for approximation errors for the proposed algorithms.

\begin{lemma}\label{lemma:my}
Let $f$ be an $(L_0,L_1)$-smooth function. Then, the following inequalities hold for any $x, y \in \mathbb{R}^d$:
\begin{align}
&\| \nabla f(y) - \nabla f(x) \| 
\le \frac{1}{L_1}\left( (L_0 + L_1 \| \nabla f(x) \|) \exp(L_1 \| y - x \|) - 1\right), \\&
| f(y) - f(x) - \langle \nabla f(x), y - x \rangle | 
\le (L_0 + L_1 \| \nabla f(x) \|) \frac{\phi(L_1 \| y - x \|)}{L_1^2},
\end{align}
where $\phi(t) = e^t - t - 1$.
\end{lemma}
\begin{proof}
    See Subsection \ref{proof:my_lemma}
\end{proof}
This result improves upon the estimation \ref{lemma:techlemma1}, since $1/L_1^2$ can be arbitrarily smaller than the distance between two points. Moreover, this result strictly better than one in \cite{vankov2024optimizingl0l1}. While this  work arrives at a structurally identical inequality, it relies on the stronger assumption that $f$ is twice differentiable. In contrast, our derivation requires only $f$ to be once differentiable.

Finally, we assume that a method has access restrictions, as discussed in the Section \ref{sec:introduction}.

\begin{assumption}[Estimate of a normalized $\nabla f(x)$]
\label{assum:approx}
The method only accesses $\tilde g$, a unit vector approximating the gradient direction such that
\begin{equation*}
\left\| \tilde g(x) - \frac{\nabla f(x)}{\|\nabla f(x)\|} \right\| \leq \delta, \quad \delta \in [0, 1),\quad \left\|\tilde g(x)\right\|=1.
\end{equation*}
\end{assumption}

\section{Normalized Gradient Descent}\label{ngd}
In this section, we consider solving problem \ref{eq:problem} using a natural generalization of the standard Gradient Descent - Normalized Gradient Descent (NGM) under quasi-convexity assumption with fixed the number of iterations $N$ before running the method. Fixing a number of iterations in advance allows to overcome an issue with extra logarithmic factor in the final estimation on the steps \cite{vankov2024optimizingl0l1}. This technique is standard in the context of normalized (sub)gradient methods (see Section 3.2 in \cite{Nesterov2018}).

\begin{algorithm}[H]\label{algo:ngd}
\caption{$(L_0,L_1)-$Inexact Normalized Gradient Descent} \label{alg:PolyakAlgo}
\begin{algorithmic}[1]
\Require Starting point $x_0$.
\For{$k = 0, \dots ,N - 1$}
    \State $x_{k+1} \gets x_k - \beta_k\gk$
\EndFor
\State \Return $x_N$
\end{algorithmic}
\end{algorithm}

The following result establishes the convergence guarantees of NGM by providing conditions on the approximation error under which the method converges to an arbitrary accuracy level $\varepsilon>0$. We highlight that these conditions require the approximation errors to be bounded merely by $\mathcal{O}(\sqrt{\varepsilon})$, i.e. $\delta$ can be much greater than the desired accuracy level, which allows for a wide range of admissible approximations while still ensuring convergence.


\begin{theorem}\label{theo:ngd}
    Suppose that Assumptions \ref{assum:quasi-convexity}--\ref{assum:approx} are satisfied. Consider the constant step-size coefficients defined by
    \[
    \beta_k = \frac{\hat{\beta}}{\sqrt{N+1}}, \qquad 0 \le k \le N-1,
    \]
    where $\hat \beta < \sqrt{3}R_0$ is a tunable parameter and $N \ge 1$ denotes the total number of iterations, prescribed in advance.
    Then, for any given $\varepsilon > 0$, the bound 
    \(
    \min_{0 \le k \le N} f(x_k) - f^* \le \varepsilon
    \)
    is guaranteed whenever
    \[
    N + 1 \ge \max\left\{ \frac{4L_1^2}{9}, \frac{L_0}{\varepsilon} \right\}
    \left( \frac{R_0^2 + 4 \hat{\beta}^2}{\hat{\beta}^2} \right)^2 =: Q(\varepsilon),
    \]

    with \[\delta \le \frac{1}{4R_0} \min \left\{ \frac{3}{2L_1}, \sqrt{\frac{\varepsilon}{L_0}},\frac{3R_0^2 - \hat{\beta}^2}{2\hat{\beta}^2 + 4R_0\hat{\beta}\sqrt{Q(\varepsilon)}} \right\}, \quad R_0:=\|x_0-x^*\|.\]
\end{theorem}

\begin{proof}
    See Subsection \ref{subsec:ngd_proof}. 
\end{proof}

A comparison between Theorem \ref{theo:ngd} and Theorem 4.4 of \cite{gasnikov2024quasiconvex} highlights two theoretical advantages of our approach. First, while \cite{gasnikov2024quasiconvex} derives an $\mathcal{O}(1/\varepsilon^2)$ iteration complexity for $L$-smooth strictly quasi-convex functions, our result reduces this requirement to $\mathcal{O}(1/\varepsilon)$ rate while assuming merely $(L_0,L_1)$-smoothness and ordinary quasi-convexity. Second, rather than assuming a priori that the algorithm operates on a compact set, we rigorously prove it.

The proof of Theorem \ref{theo:ngd} is based on the following lemma that can be found, for example, in \cite{Nesterov2018} (Theorem 1.5.5).

\begin{lemma}\label{lemma:lemma_nesterov}
Let $f \colon \mathbb{R}^d \to \mathbb{R}$ be a differentiable quasi-convex function.
Then, for any $y \in \mathbb{R}^d$ and
\[
v_f(x; y) := \frac{[\langle \nabla f(x), x - y \rangle]_+}{\| \nabla f(x) \|},\qquad [t]_+=\max\{t,0\},
\]
it holds that
\[
f(y) - f(x^*) \le \max_{z \in \mathbb{R}^d} \bigl\{ f(z) - f(x^*) : \|z - x^*\| \le v_f(y; x^*) \bigr\}.
\]
\end{lemma}

These
values $v_k$ have a geometrical meaning \--- each of them is exactly the distance from the
point $x^*$
to the supporting hyperplane to the sublevel set of $f$ at the point $x_k$. The main characteristic of this type algorithms is that as $v_k$ converges to zero, so does \(
    \min_{0 \le k \le N} f(x_k) - f^*
    \), which is a core of our proof.

\section{Gradient Descent with Polyak Stepizes}
In this section, we present a theoretical analysis of Algorithm~\ref{alg:PolyakAlgo} for solving problem~\ref{eq:problem} under the $(L_0,L_1)$-smoothness assumption when only a Comparison Oracle is available. The method is based on Gradient Descent with Polyak stepsizes, a widely studied class of optimization algorithms that is relevant in settings where the optimal objective value is known. That happens, for example, when the optimizing model over-parametrized with $f^* = 0$, as shown in \cite{pmlr-v89-vaswani19a}.

Such methods have found numerous applications in machine learning and optimization, both in their classical form \cite{vankov2024optimizingl0l1,hazan2019polyak} and in various extensions, including SGD with Polyak step-sizes \cite{loizou2020SGD_polyak}, clipped Polyak \cite{takezawa2024parameterfree}, and momentum-based variants \cite{gao2024momentum_polyak}. While our framework assumes exact objective function evaluations paired with inexact gradient directions, this operational regime remains relevant in practice, where gradient computation is typically the primary computational bottleneck.

\begin{algorithm}[H]
\caption{$(L_0,L_1)-$Inexact Normalized Gradient Descent with Polyak stepsizes}\label{alg:PolyakAlgo}
\begin{algorithmic}[1]
\Require Starting point $x_0$, optimal function value $f^*$, number of iterations $N$, constant $M$ such that $M \ge \|\nabla f(x)\|$ for all $x$ along the trajectory.
\For{$k=0,\dots,N-1$}
    \State $x_{k+1} \gets x_k - \dfrac{f(x_k) - f^*}{M}\gk$
\EndFor
\State \Return $\hat x_N := \arg\min_{x<N}f(x_k)$
\end{algorithmic}
\end{algorithm}
Before proving convergence, we establish a lemma giving sufficient conditions for the existence of a uniform bound $M$ on the gradient norm over all iterations prior to achieving the target accuracy, along with an explicit upper bound.

\begin{lemma}\label{lemma:polyak}Suppose Assumptions \ref{assum:smoothness}--\ref{assum:approx} and \ref{assum:convexity} hold with $\mu = 0$. Let $\varepsilon >0$, $\nu e^\nu = 1$ and $R_0 = \|x_0 - x^*\|$. Then, for any $c \in \left(0, \varepsilon/R_0\right)$, if $\delta$ satisfies
\begin{equation}\label{e:polyak_delta_bound}
    \delta \le \frac{\nu c}{16R_0}\min \left\{ \frac{1}{L_0}, \frac{\nu}{4L_1(\max_{k<N}f(x_k)-f^*)}\right\},
\end{equation}

at least one of the following holds:
\begin{enumerate}
\item For all $k \ge 0$ it holds that $\nnfk > c$, the distance to the optimum strictly decreases $\Rkn < \Rk$, and the gradient is bounded above by the constant $\|\nabla f(x_k)\| < M_{\max}$, such that $M_{\max}\le \frac{L_0}{L_1}\left(\exp(L_1R_0)-1\right)$.

\item There exists $k^*$, such that $\|\nabla f(x_{k^*})\| \le c$, and $f(\hat x_N)-f^*\le f(x_{k^*})-f^*\le\varepsilon$.
\end{enumerate}\end{lemma}

\begin{proof}
    See Subsection \ref{proof:lemma_polyak}.
\end{proof}

Consequently, for any threshold $\varepsilon$ with appropriate gradient bound $c$ and a sufficiently small $\delta$, either the gradient norm is bounded from both below and above while the distance to the optimum is decreasing, or the desired suboptimality gap is achieved.

\begin{remark}
In Algorithm~\ref{alg:PolyakAlgo}, an alternative stopping criterion may be employed, eliminating the need to specify the total number of iterations in advance. In this setting, the quantity \(\max_{k<N} f(x_k)-f^*\) is no longer available. Therefore, the condition in~\ref{e:polyak_delta_bound} should be replaced with
\begin{equation}\label{e:polyak_delta_bound2}
    \delta \le \frac{\nu c}{16L_0R_0}\min \left\{ 1, \frac{\nu}{2}\max\left\{\frac{1}{L_1\exp(L_1R_0)}, \frac{L_0L_1}{2\phi(L_1R_0)}\right\} \right\}.
\end{equation}
\end{remark}
\begin{proof}
    See Subsection \ref{proof:lemma_polyak}.
\end{proof}
Now, we can provide the main result of the current section, which is similar to one in \cite{gorbunov2024methodsl0l1}.

\begin{theorem}\label{theo:polyak}
Suppose Assumptions \ref{assum:smoothness}--\ref{assum:approx} and \ref{assum:convexity} hold with $\mu = 0$. Let $\varepsilon >0$, $\nu e^\nu = 1$, $R_0 = \|x_0 - x^*\|$, and $\delta$ satisfies condition \ref{e:polyak_delta_bound}. Then, unless an $\varepsilon$-accuracy is achieved, the following properties hold:
\begin{enumerate}
    \item \textbf{Per-step convergence:} If $c\le\frac{L_0}{L_1}$, then at each iteration $k \ge 0$, we have
    \begin{equation}\label{eq:polyak_1_st}
    \begin{aligned}
                        \|\nabla f(x_k)\| \ge \frac{L_0}{L_1} &\implies \|x_{k+1}-x^*\| ^2\le \|x_k-x^*\|^2 - \frac{\nu^2 c}{32 M L_1^2}, \\[1ex]
        \|\nabla f(x_k)\| < \frac{L_0}{L_1} &\implies \|x_{k+1}-x^*\|^2 \le \|x_k-x^*\|^2 - \frac{\nu c (f(x_k)-f^*)}{8 M L_0}.
    \end{aligned}
    \end{equation}

    Otherwise, if $c>\frac{L_0}{L_1}$, then at each iteration $k \ge 0$, we have
    \[
    \|x_{k+1}-x^*\| ^2\le \|x_k-x^*\|^2 - \frac{\nu^2 c}{32 M L_1^2}.
    \]
    \item \textbf{Optimality gap:} If $c<\frac{L_0}{L_1}$ and $N > \frac{32 R_0^2 M L_1^2}{\nu^2 c} - 1$, then
    \begin{equation}\label{eq:polyak_second_st}
       f(\hat x_N) - f^* \le \frac{8 M L_0 R_0^2}{\nu (N + 1)}.
    \end{equation}
    \item \textbf{Distance linear convergence:} If $c<\frac{L_0}{L_1}$ and Assumption \ref{assum:convexity} holds with $\mu > 0$, then
    \begin{equation}
        \|x_N - x^*\|^2 \le \left( 1 - \frac{\mu \nu c}{16 M L_0} \right)^{N-P} R_0^2,
    \end{equation}
    where $P := |\{0 \le k \le N : \|\nabla f(x_k)\| \ge L_0/L_1\}|$.
\end{enumerate}
\end{theorem}
\begin{proof}
    See Subsection \ref{subsec:theo_polyak}.
\end{proof}


Proof of the Theorem \ref{theo:polyak} relies on the following lemma, established in \cite{gorbunov2024methodsl0l1}.
\begin{lemma}
\label{lemma:techlemma2}
Let $f$ be convex $(L_0,L_1)$-smooth function and $\nu$ satisfy $\nu = e^{-\nu}$. Then, for any $x \in \mathbb{R}^d$, we have
    \begin{equation*}
        \frac{\nu \|\nabla f(x)\|^2}{2(L_0 + L_1 \|\nabla f(x)\|)} \le f(x) - f^*.
    \end{equation*}
\end{lemma}

Although Lemma \ref{lemma:polyak} provides an explicit uniform bound on $M$, this theoretical value can be excessively large in practice, leading to overly conservative step sizes and slow convergence. To mitigate this, if prior estimates of $L_0$ and $L_1$ are available, the global constant $M$ can be replaced by a dynamic local estimate $M_k$. By inverting the inequality from Lemma \ref{lemma:techlemma2}, the current suboptimality gap $f(x_k) - f^*$ directly constrains the maximum possible value of $\|\nabla f(x_k)\|$. Solving this relation at each iteration yields a tighter, adaptive upper bound $M_k$, taken as the largest root of the corresponding equation. This leads to the following adaptive variant of Algorithm~\ref{alg:PolyakAlgo}.

\begin{algorithm}[H]
\caption{$(L_0,L_1)$-Gradient Descent with Adaptive Polyak Stepsizes for Comparison Oracle}
\label{alg:AdaptivePolyakAlgo}
\begin{algorithmic}[1]
\Require Starting point $x_0$, optimal function value $f^*$, number of iterations $N$, constants $L_0$, $L_1$.
\For{$k=0,\dots,N-1$}
    \State Compute $M_k$ as the largest solution of
    \[
        \frac{\nu M_k^2}{2(L_0+L_1M_k)}
        = f(x_k)-f^*
    \]
    \State $x_{k+1} \gets x_k - \dfrac{f(x_k)-f^*}{M_k}\gk$
\EndFor
\State \Return $\hat x_N := \arg\min_{k<N} f(x_k)$
\end{algorithmic}
\end{algorithm}

At present, we do not provide a theoretical convergence analysis for this adaptive variant. Nevertheless, as demonstrated in the numerical experiments (denoted as \emph{Adaptive} on the graphs), it consistently outperforms the version based on the uniform bound \(M\), suggesting that the adaptive estimate provides a substantially less conservative approximation of the gradient norm.

\section{Numerical Experiments}

In this section, we verify the theoretical findings through numerical experiments on the synthetic convex function
\[
\hat F:\mathbb{R}^d\to\mathbb{R},\qquad
\hat F(x)=\frac{L_0L_1^2}{24}\sum_{i=1}^d x_i^4+\frac{L_0}{2}\sum_{i=1}^d x_i^2+f_{\mathrm{opt}},
\]
which is $(L_0,L_1)$-smooth  for any \(L_0,L_1>0\) (see proof in the Subsection \ref{proof:l0l1}) and has the global optimum \(f^*=f_{\mathrm{opt}}\). The experiments reported in Figure~1 were conducted with \(L_0=10\) and \(f_{\mathrm{opt}}=1\), while varying the parameters \(L_1\), \(\delta\), and the initial point $x_0$. For the Polyak-type method, the parameter \(M\) was set to \(10L_1\) in all experiments. For Normalized Gradient Descent, changing parameter $\hat\beta$ did not significantly affect convergence, and the final was one set to $0.1$. 

The empirical results are consistent with the theoretical predictions. In particular, the convergence rate of NGD is highly sensitive to the choice of the initial point, as illustrated in the upper row of Figure~1, whereas the Polyak-type method exhibits substantially weaker dependence on initialization. At the same time, the Polyak-type method demonstrates faster practical convergence, although, in agreement with the theory, its objective values are not necessarily monotone. Finally, both algorithms are noticeably affected by the parameter \(L_1\), while increasing the approximation error \(\delta\) leads to the expected deterioration in the convergence behavior.

\begin{figure}[htbp]
    \centering
    \includegraphics[width=1\textwidth]{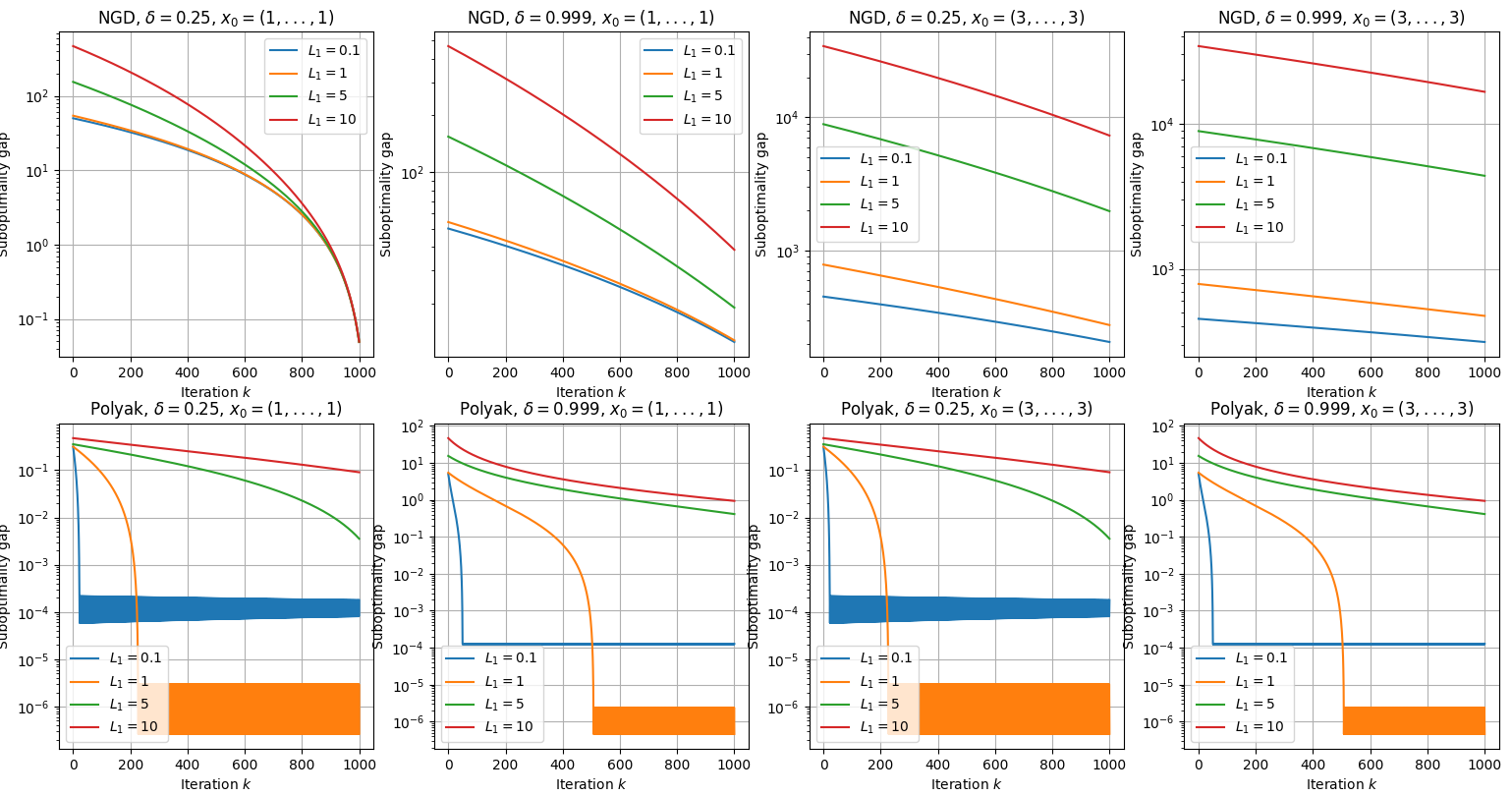}
    \caption{Convergence behavior of NGD (top row) and the Polyak method (bottom row) on the synthetic \((L_0,L_1)\)-smooth objective function \(\hat F\) for different values of \(L_1\), initialization point \(x_0\), and approximation error \(\delta\).}
    \label{fig:}
\end{figure}

\begin{figure}[htbp]\label{fig:fig2}
    \centering
    \includegraphics[width=1\textwidth]{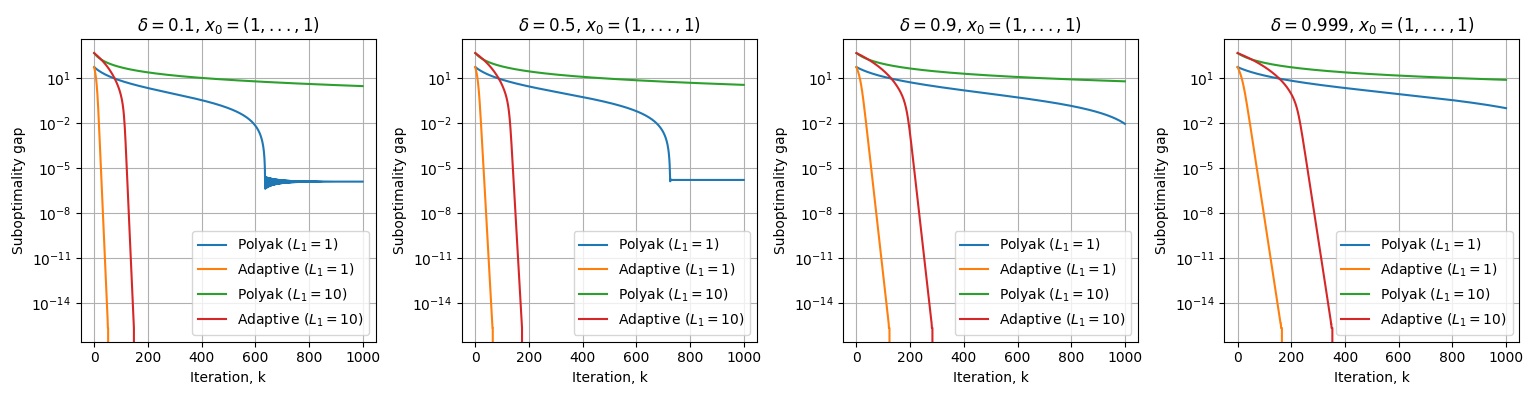}
    \caption{Convergence behavior of Polyak and adaptive Polyak on the synthetic \((L_0,L_1)\)-smooth objective function \(\hat F\) for different values of \(L_1\) and approximation error \(\delta\).}
    \label{fig:}
\end{figure}
\bibliographystyle{splncs04}
\bibliography{bib}

\section{Missing proofs}

\subsection{Proof of Lemma \ref{lemma:my}}\label{proof:my_lemma}

The following derivation adheres to the framework established in \cite{chen2023generalizedsmooth}, extending their results to derive our primary contribution.

We define the auxiliary function $H(\theta)$ as$$H(\theta) := L_0 \theta + L_1 \int_{0}^{\theta} \|\nabla f(z_u)\| \, du,$$
where $z_u := (1-u)x + uy$. As demonstrated in \cite{chen2023generalizedsmooth}, $H(\theta)$ satisfies the differential inequality:
\begin{equation*}
H'(\theta) \le L_0 + L_1 \| y - x \| H(\theta) + L_1 \| \nabla f(x) \|.
\end{equation*}
Applying Grönwall's inequality yields:
\begin{equation*}
L_1\|y-x\|H(\theta) \le (L_0 + L_1\|\nabla f(x)\|)(\exp(L_1\|y-x\|\theta) - 1).
\end{equation*}
Invoking Lemma A.1 from \cite{chen2023generalizedsmooth}, we bound the gradient variation by:
\begin{equation*}
\| \nabla f(y) - \nabla f(x) \| \le \left( L_0 + L_1 \int_0^1 \| \nabla f(z_\theta) \| \, d\theta \right) \| y - x \|.
\end{equation*}
Recognizing that the term in the parentheses corresponds to $H(1)$, we substitute the Grönwall estimate to obtain:
\begin{equation}\label{eq_new1}
\| \nabla f(y) - \nabla f(x) \| \le \frac{1}{L_1}\left( (L_0 + L_1 \| \nabla f(x) \|) \exp(L_1 \| y - x \|) - 1\right).
\end{equation}
At this point, we have proved the first statement of Lemma \ref{lemma:my}. 
The final estimation is derived through the following integration:
\begin{align*}
f(y) - f(x) - \langle \nabla f(x), y - x \rangle 
&\le \int_0^1 \| \nabla f(z_\theta) - \nabla f(x) \| \| y - x \| \, d\theta \\
&= \frac{\| y - x \|}{L_1} (L_0 + L_1 \| \nabla f(x) \|) \int_0^1 (\exp(\theta L_1 \| y - x \|) - 1) \, d\theta \\
&= (L_0 + L_1 \| \nabla f(x) \|) \frac{\phi(L_1 \| y - x \|)}{L_1^2}.
\end{align*}
\subsection{Proof of Theorem \ref{theo:ngd}}\label{subsec:ngd_proof}

The iteration step of the algorithm can be rewritten as follows:
\begin{equation*}
    x_{k+1} = x_k - \beta_k \frac{\nfk}{\nnfk} - \beta_k \underbrace{\left( \gk - \fnfk \right)}_{e_k}, \quad \|e_k\| \le \delta.
\end{equation*}

Next, we expand the squared distance to the optimal solution $x^*$:
\begin{equation}\label{eq:ngm_main}
\begin{aligned}
    \Rkn^2 &= \left\| x_k - x^* - \beta_k \frac{\nfk}{\nnfk} - \beta_k e_k \right\|^2 \\
    &\le \Rk^2 + \beta_k^2(1 + \delta^2)
     + 2\beta_k \inner{e_k}{x^* - x_k}
     \\&+ 2\beta_k^2 \inner{\fnfk}{e_k}
     - 2\beta_k \underbrace{\inner{\fnfk}{x_k - x^*}}_{v_k} \\
    &\le \Rk^2 - 2\beta_k v_k + 2\Rk \beta_k \delta + \beta_k^2(1 + \delta)^2.
\end{aligned}
\end{equation}

Defining $R_k := \|x_k - x^*\|$, we proceed by induction to show $R_k \le 2R_0$ for all $k \le N$.
The base case $k=0$ holds trivially ($R_0 \le 2R_0$).
Assume $R_k \le 2R_0$ holds for all $k < m \le N$. We now show $R_m \le 2R_0$.

Summing inueqalities \ref{eq:ngm_main} and applying quasi-convexity yields:
\begin{equation*}
    R^2_{m} \le R_0^2 + 2\delta \sum_{k=0}^{m-1} \beta_k R_k + (1 + \delta)^2 \sum_{k=0}^{m-1} \beta_k^2<R_0^2+4\delta R_0\sqrt{N+1}\hat\beta+(1+\delta)^2\hat\beta^2.
\end{equation*}
To guarantee $R_m \le 2R_0$, it suffices to show the right-hand side is bounded by $4R_0^2$. We define the corresponding condition as a quadratic function in $\delta$:$$\mathcal{E}(\delta) = \underbrace{\hat{\beta}^2}_{a}\delta^2 + \underbrace{(2\hat{\beta}^2 + 4R_0\sqrt{N+1}\hat{\beta})}_{b}\delta + \underbrace{\hat{\beta}^2 - 3R_0^2}_{c}.$$Since $a > 0$ and $b > 0$, the parabola opens upward with its vertex in the negative half-plane ($-\frac{b}{2a} < 0$). Provided $\mathcal{E}(0) = c < 0$ (i.e., $\hat{\beta}^2 < 3R_0^2$), $\mathcal{E}(\delta) \le 0$ holds for all $\delta$ up to the positive root $\delta_+$. By bounding $\delta_+$ from below, we establish a sufficient upper bound for $\delta$:$$\delta_+ = \frac{-b + \sqrt{b^2 - 4ac}}{2a} = \frac{-2c}{b + \sqrt{b^2 - 4ac}} > \frac{-2c}{2b} = -\frac{c}{b} = \frac{3R_0^2 - \hat{\beta}^2}{2\hat{\beta}^2 + 4R_0\sqrt{N+1}\hat{\beta}}.$$Therefore, choosing any $\delta \le -\frac{c}{b}$ ensures $\mathcal{E}(\delta) \le 0$, which yields $R_m^2 \le 4R_0^2$, implying $R_m \le 2R_0$. This completes the induction.

Next, we will get back to equation \ref{eq:ngm_main}.
Summing these inequalities from $k = 0$ to $N$, we obtain:
\begin{equation*}
    R_{N+1} \le R_0^2 - 2\sum_{k=0}^N \beta_k v_k + 2\delta \sum_{k=0}^N \beta_k R_k + (1 + \delta)^2 \sum_{k=0}^N \beta_k^2.
\end{equation*}

Introducing the notation $v^*_N := \min_{0 \le k \le N} v_k$,we deduce:
\begin{equation*}
    v^*_N \le \frac{R_0^2 + 4\delta R_0 \sum_{k=0}^N \beta_k + (1 + \delta)^2 \sum_{k=0}^N \beta_k^2}{2\sum_{k=0}^N \beta_k} 
    \le \underbrace{\frac{R_0^2 + (1 + \delta)^2 \sum_{k=0}^N \beta_k^2}{2\sum_{k=0}^N \beta_k}}_{S_1} + \underbrace{2\delta R_0}_{S_2}.
\end{equation*}

Let $f_N^* := \min_{0 \le k \le N} f(x_k)$ and $k_0 = \arg\min_{0 \le k \le N} v_k$, so $v_{k_0} = v^*_N$. By definition of $f^*_N$, we have $f^*_N \le f(x_{k_0})$. Applying Lemma \ref{lemma:lemma_nesterov}, we have:
\begin{equation*}
    f_N^* - f^* \le f(x_{k_0})-f^*\le\max_{z} \bigl\{ f(z) - f^* : \|z - x^*\| \le v_N^* \bigr\}.
\end{equation*}

Using Lemma \ref{lemma:my} and the property $\phi(t) \le \frac{3t^2}{6 - 2t}$ for all $t \in [0, 3)$ (see \cite[Lemma 2.6]{vankov2024optimizingl0l1}), it follows that:
\begin{equation*}
    f_N^* - f^* \le \frac{L_0}{L_1^2} \phi(L_1 v_N^*) \le \frac{3L_0 (v_N^*)^2}{6 - 2L_1 v_N^*}.
\end{equation*}

To guarantee convergence, we require $v_N^* \le \frac{3}{2L_1}$ and $v_N^* \le \sqrt{\frac{\varepsilon}{L_0}}$. These bounds are satisfied provided that both terms $S_1$ and $S_2$ do not exceed
\begin{equation*}
    \frac{\mathcal{M}}{2}, \quad \text{where} \quad \mathcal{M} := \min\left\{\frac{3}{2L_1}, \sqrt{\frac{\varepsilon}{L_0}}\right\}.
\end{equation*}
Equivalently, the above conditions reduce to the following two inequalities:
\begin{equation*}
        \dfrac{R_0^2 + (1 + \delta)^2 \sum_{k=0}^N \beta_k^2}{2\sum_{k=0}^N \beta_k} \le \dfrac{\mathcal{M}}{2};\quad\quad\quad
        2\delta R_0 \le \dfrac{\mathcal{M}}{2}.
\end{equation*}
The second condition obviously bounds $\delta$. Using a definition of $\hat \beta_k$, the first inequality gives:
\begin{equation*}
    \frac{R_0^2 + (1 + \delta)^2 \hat{\beta}^2}{2\hat{\beta}\sqrt{N+1}} \le \frac{\mathcal{M}}{2}.
\end{equation*}

Rearranging for the total number of iterations $N$, we obtain:
\begin{equation*}
    N + 1 \ge \left( \frac{R_0^2 + (1 + \delta)^2 \hat{\beta}^2}{\mathcal{M} \hat{\beta}} \right)^2 
    = \max\left\{\frac{4L_1^2}{9}, \frac{L_0}{\varepsilon}\right\} \left( \frac{R_0^2 + (1 + \delta)^2 \hat{\beta}^2}{\hat{\beta}} \right)^2.
\end{equation*}
Bounding $\delta$ from above by $1$, we get the final estimation.

Finally, minimizing the function $\hat{\beta} \mapsto \frac{R_0^2 + 4 \hat{\beta}^2}{\hat{\beta}}$ by using AM-GM inequality, we find the optimal parameter choice to be $\hat{\beta} = \frac{R_0}{2}$.

\subsection{Proof of Lemma \ref{lemma:polyak}}\label{proof:lemma_polyak}

First, assume that $\nnfk>c$. We will bound the squared distance from $x_{k+1}$ to an optimal point:
\begin{align*}
    \|x_{k+1} - x^*\|^2 &= \left\| x_k - x^* - \frac{\fk - f^*}{M} \gk \right\|^2 \\
    &= \|x_k - x^*\|^2 - \frac{2(\fk - f^*) \langle x_k - x^*, \gk \rangle}{M} + \left(\frac{\fk-f^*}{M}\right)^2.
\end{align*}

Now, we will expand the second term:
\begin{align*}
    \langle x_k - x^*, \gk \rangle &= \left\langle x_k - x^*, \frac{\nabla \fk}{\|\nabla \fk\|} \right\rangle + \left\langle x_k - x^*, \gk - \frac{\nabla \fk}{\|\nabla \fk\|} \right\rangle \\
    &\ge \left\langle x_k - x^*, \frac{\nfk}{\nnfk} \right\rangle - \delta\|x_k-x^*\|,
\end{align*}
where the last inequality follows directly from Assumption \ref{assum:approx}.

Combining the last two inequalities, we obtain:
\begin{equation}\label{eq:polyak_to}
    \begin{aligned}
            \|x_{k+1} - x^*\|^2 &\le \|x_k - x^*\|^2 - \frac{2(\fk - f^*) \langle x_k - x^*, \nfk \rangle}{M\nnfk} \\
    &\quad + \frac{2(\fk-f^*)\|x_k-x^*\|\delta}{M} + \left(\frac{\fk-f^*}{M}\right)^2 \\
    &\le \Rk^2-\frac{2(\fk-f^*)^2}{M\nnfk}+\\&+\frac{2(\fk-f^*)\|x_k-x^*\|\delta}{M}+\left(\frac{\fk-f^*}{M}\right)^2  \\
    &\le \Rk^2-\frac{c(\fk-f^*)^2}{M\nnfk^2}+\frac{2(\fk-f^*)\|x_k-x^*\|\delta}{M}\\
    &\le \Rk^2-\frac{c(\fk-f^*)\nu}{2M(L_0+L_1\nnfk)}+\frac{2(\fk-f^*)\|x_k-x^*\|\delta}{M},
    \end{aligned}
\end{equation}

where we use boundedness of the gradient from below and Lemma \ref{lemma:techlemma2} in the last two inequalities.

We divide the proof into two cases based on the gradient norm: $\nnfk \ge \frac{L_0}{L_1}$ and $\nnfk < \frac{L_0}{L_1}$. 

Without loss of generality, we assume $c < \frac{L_0}{L_1}$, since otherwise the first case directly implies the second statement of Lemma \ref{lemma:polyak}, and the second case remains unchanged with the condition strengthened $\nnfk \ge \frac{L
_0}{L_1}>c$.

\textbf{Case 1:} $\nnfk < \frac{L_0}{L_1}$.

As easy to see, in that case, $L_0+L_1\nnfk<2L_0$.
Then, from \ref{eq:polyak_to} we have an estimation:
\begin{align}\label{eq:delta_bound2}
    \|x_{k+1}-x^*\|^2\le \Rk^2-\frac{c\nu(\fk-f^*)}{4ML_0}+\frac{2(\fk-f^*)\Rk\delta}{M}.
\end{align}
Now, to guaranty monotonic decrease of distance from $k-$th point to an optimal one, we require the next boundedness of $\delta$:
\begin{align}\label{eq:polyak_first_delta_bound}
    \delta \le \frac{c\nu}{16L_0\Rk}.
\end{align}
\textbf{Case 2:} $\nnfk \ge \frac{L_0}{L_1}$.
In this case, the following inequality obviously holds: 
\begin{equation}\label{eq:obv}
    L_0 +L_1\nnfk\le 2L_1\nnfk,
\end{equation}
which, with help of Lemma \ref{lemma:techlemma2}, induces the next estimation:
\begin{equation}\label{eq:tech_lemma2}
    \frac{\nu\nnfk}{4}\le\frac{\nu\nnfk^2}{2(L_0+L_1\nnfk)}\le \fk-f^*.
\end{equation}

Now, we can estimate the squared distance at the $(k+1)$-th iteration more precisely:
\begin{equation}\label{eq:delta_bound1}
    \begin{aligned}
        \|x_{k+1}-x^*\|^2 &\le \Rk^2 - \frac{c\nu(\fk-f^*)}{2M(L_0+L_1\nnfk)} + \frac{2(\fk-f^*)\|x_k-x^*\|\delta}{M} \\
        &\le \Rk^2 - \frac{c\nu(\fk-f^*)}{4M\nnfk L_1} + \frac{2(\fk-f^*)\|x_k-x^*\|\delta}{M} \\
        &\le \Rk^2 - \frac{c\nu^2}{16ML_1} + \frac{2(\fk-f^*)\|x_k-x^*\|\delta}{M},
    \end{aligned}
\end{equation}
where in the last two inequalities we consequently used inequalities \ref{eq:obv} and \ref{eq:tech_lemma2}.

As in the previous case, to ensure that the distance $R_k :=\|x_k-x^*\|$ decreases, we bound $\delta$:
\begin{align}\label{eq:polyak_second_delta_bound}
    \delta \le \frac{c\nu^2}{64L_1(\max_{k<N}f(x_k)-f^*)R_k}\le \frac{c\nu^2}{64L_1(f(x_k)-f^*)R_k}.
\end{align}

In order to make the bounds independent of $N$, we can estimate suboptimality gap from above using Lemmas \ref{lemma:my} and \ref{lemma:techlemma1} with $y=x_k$ and $x=x^*$ as follows:
\[
f(x_k)-f^*\le \dfrac{1}{L_1^2}\phi(L_1R_k);\quad\quad f(x_k)-f^*\le \dfrac{L_0}{2}\exp({L_1R_k})R_k^2.\]

Then, we have to require $\delta$ to satisfy 
\begin{align}\label{eq:tocite1}
\delta&\le \dfrac{c\nu^2}{32R_k}\max\{\frac{1}{L_0L_1\exp\left(L_1R_k\right)R_k^2},\,\, \frac{L_1}{2\phi(L_1R_k)}\}.
\end{align}

That is, if both of inequalities \ref{eq:polyak_first_delta_bound} and \ref{eq:polyak_second_delta_bound} (or \ref{eq:polyak_first_delta_bound} and \ref{eq:polyak_second_delta_bound} in setting with no dependency to $N$) hold, then, by induction, the sequence $R_k$  monotonically decreases, and in both requirements we can change $R_k$ to $R_0$.

Nevertheless, since the number of iterations is fixed in our setting, the original bound is sufficient. Moreover, it is typically much less conservative than the $N$-independent estimate above and is therefore more practical to use.

It follows that the iterates remain in the compact ball
\(
\mathcal{X}:=\{x\in\mathbb R^d:\|x-x^*\|\le \|x_0-x^*\|\}
\). Since $\|\nabla f(x)\|$ is continuous, the Weierstrass theorem guarantees there is a constant $M_{max} > 0$ such that $\|\nabla f(x)\| \le M_{max}$ for all $x \in \mathcal{X}$. Therefore, $\|\nabla f(x_k)\| < M_{max}$ for all \(k\) where $\nnfk > c$. Finally, by setting $y = x$ and $x=x^*$ in the first statement of Lemma \ref{lemma:my}, we obtain the upper bound for the constant $M_{max}$. Notice that during the proof we did not rely on any assumptions about $M$. And only now it becomes possible to tell that this constant is theoretically bounded.

Next, we will prove the second statement. Assume that $\|\nabla f(x_{k})\| > c$ for the first time during iterations. Then, at the previous iteration, $\|\nabla f(x_{k-1})\|\le c$, and as we already proved, $\Rk\le\|x_{k-1}-x^*\|$. By exploiting convexity, we deduce that $f(x_k) - f^* \le \inner{\nfk}{x_k-x^*}\le\nnfk\Rk\le\varepsilon$, thereby concluding the proof of Lemma \ref{lemma:polyak}.

\subsection{Proof of Theorem \ref{theo:polyak}}\label{subsec:theo_polyak}
In this subsection, we will use the definitions of Subsection \ref{proof:lemma_polyak}.

Substituting bounds of $\delta$, derived in the Section \ref{proof:lemma_polyak} into inequalities \ref{eq:delta_bound2} and \ref{eq:delta_bound1} respectively, we prove the first statement of Theorem \ref{theo:polyak}.

Next, define the index sets $\mathcal{N} := \{0, 1, \dots, N\}$ and $\mathcal{P} := \left\{k \in \mathcal{N} : \nnfk \ge \frac{L_0}{L_1}\right\}$, and let $P := |\mathcal{P}|$. 
From \eqref{eq:polyak_1_st}, we have
\begin{equation*}
    \|x_{N+1}-x^*\|^2 \le R_0^2 - \frac{\nu^2 c P}{32 M L_1^2} - \frac{\nu c}{8 M L_0}\sum_{k \in \mathcal{N} \setminus \mathcal{P}} (\fk - f^*).
\end{equation*}
Rearranging this inequality yields
\begin{equation}\label{eq:eqned}
    \frac{8 M L_0}{\nu c}\|x_{N+1}-x^*\|^2 + \sum_{k \in \mathcal{N} \setminus \mathcal{P}} (\fk - f^*) \le \frac{8 M L_0 R_0^2}{\nu c} - \frac{\nu L_0 P}{4L_1^2}.
\end{equation}
Since the left-hand side of \eqref{eq:eqned} is non-negative, we deduce the bound
\[
    P \le \frac{32 R_0^2 M L_1^2}{\nu^2 c}.
\]
Consequently, choosing a number of iterations  
\[
    N > \frac{32 R_0^2 M L_1^2}{\nu^2 c} - 1
\]
guarantees that $N > P$. This ensures that the set $\mathcal{N} \setminus \mathcal{P}$ is non-empty, allowing us to bound the optimality gap at $\hat{x}_N$ as follows:
\begin{equation*}
    (N - P + 1)(f(\hat{x}_N) - f^*) \le \sum_{k \in \mathcal{N} \setminus \mathcal{P}} (\fk - f^*) \le \frac{8 M L_0 R_0^2}{\nu c} - \frac{\nu L_0 P}{4 L_1^2}, 
\end{equation*}
or, dividing by $(N-P+1)$:
\begin{equation}\label{eq:eq123}
    f(\hat{x}_N) - f^* \le \frac{8 M L_0 R_0^2}{\nu c(N - P + 1)} - \frac{\nu L_0 P}{4 L_1^2(N - P + 1)}, 
\end{equation}
Now, to prove \eqref{eq:polyak_second_st}, we will consider the right-hand side of \eqref{eq:eq123} as a function with respect to $P$, which can be rewritten in a more convenient form:

\begin{align*}
    \varphi(P) &:=  \frac{8ML_0R_0^2}{c\nu(N - P + 1)} + \frac{\nu L_0}{4L_1^2} - \frac{\nu L_0(N + 1)}{4L_1^2(N - P + 1)}.
\end{align*}
The derivative of this function is the following
\begin{equation*}
    \varphi'(P) = \frac{8ML_0R_0^2}{c\nu(N - P + 1)^2} - \frac{\nu L_0(N + 1)}{4L_1^2(N - P + 1)^2}.
\end{equation*}
Since \(    N > \frac{32 R_0^2 M L_1^2}{\nu^2 c} - 1
\), we have $\varphi'(P)<0$, which means that $\varphi$ is a decreasing function. That implies that for $P>0$ we have $\varphi(P)\le\varphi(0)$, that is,
\[
\frac{8ML_0R_0^2}{c\nu(N - P + 1)} - \frac{\nu L_0 P}{4L_1^2(N - P + 1)}\le\frac{8ML_0R_0^2}{\nu(N+1)}.
\]
Combining the last inequality with \eqref{eq:eq123}, we establish the second result of Theorem \ref{theo:polyak}.

Now, to prove the third statement of Theorem \ref{theo:polyak}, we consider equation \ref{assum:convexity} to hold with $\mu > 0$. In that case, obviously, \(\fk - f^*\ge \frac{\mu}{2}\Rk^2\), which implies
\[
\|x_{k+1}-x^*\|^2\le \left(1-\frac{\mu\nu c}{16ML_0}\right)\Rk^2,
\]
when $\nnfk \le \frac{L_0}{L_1}$.

Then we can bound the distance between the $N-$th point and the optimal one for $N < P$ as follows:
\[
\|x_N-x^*\|^2\le\left(1-\frac{\mu\nu c}{16ML_0}\right)^{N-P}R_0^2-\frac{\nu^2c}{32ML_1^2}\sum_{k\in\mathcal{P}}{\left(1-\frac{\mu\nu c}{16ML_0}\right)^{p_k}}, 
\]
where $p_k := |\{0, \dots, k\}\setminus \mathcal{P}|$.

As $|\mathcal{P}|=P,$ then $p_k\le N-P$ for all $k \in \mathcal{P}$, which allows us to combine both terms of the right-hand side of the previous inequality as follows:
\[
\|x_N-x^*\|^2\le\left(1-\frac{\mu\nu c}{16ML_0}\right)^{N-P}\left(R_0^2-\frac{\nu^2cP}{32ML_1^2}\right)\le\left(1-\frac{\mu\nu c}{16ML_0}\right)^{N-P}R_0^2, 
\]
which finished the proof of Theorem \ref{theo:polyak}.

\subsection{Proof that Function $\hat F$ is $(L_0,L_1)$-smooth for any $L_0,L_1 > 0$}\label{proof:l0l1}

First, project onto the $i$-th coordinate axis by setting all coordinates except $i$ to zero. For the Hessian of this projection, we have
\[
|\nabla^2\hat F_i(x)| = \frac{L_0L_1^2}{2}x_i^2+L_0\le L_1|x_i|\underbrace{\left(\dfrac{L_0L_1^2}{6}x_i^2+L_0\right)}_{\ge2\sqrt{L_0^2L_1^2x_i^2/6}\ge L_0L_1|x_i|/2}+L_0=L_1|\nabla \hat F_i(x)|+L_0.
\]
Next, using Proposition 2.4.2 in \cite{vankov2024optimizingl0l1}, we complete the proof.
\end{document}

%% file: math_commands.tex
\usepackage{amsmath,amsfonts,bm}

\def\eqref#1{equation~\ref{#1}}

\def\1{\bm{1}}

\DeclareMathAlphabet{\mathsfit}{\encodingdefault}{\sfdefault}{m}{sl}
\SetMathAlphabet{\mathsfit}{bold}{\encodingdefault}{\sfdefault}{bx}{n}



%% file: bib.bib
@book{Nesterov2018,
  author    = {Yurii Nesterov},
  title     = {Lectures on Convex Optimization},
  series    = {Springer Optimization and Its Applications},
  volume    = {137},
  publisher = {Springer},
  edition   = {2},
  year      = {2018},
  doi       = {10.1007/978-3-319-91578-4},
  isbn      = {978-3-319-91577-7}
}

@article{nesterov2017gradientfree,
  author  = {Nesterov, Yurii and Spokoiny, Vladimir},
  title   = {Random Gradient-Free Minimization of Convex Functions},
  journal = {Foundations of Computational Mathematics},
  volume  = {17},
  number  = {2},
  pages   = {527--566},
    doi     = {https://doi.org/10.1007/s10208-015-9296-2}
,
  year    = {2017},
}

@article{gasnikov2024quasiconvex,
  title={On Quasi-Convex Smooth Optimization Problems by a Comparison Oracle},
  author={Gasnikov, A. V. and Alkousa, M. S. and Lobanov, A. V. and Dorn, Y. V. and Stonyakin, F. S. and Kuruzov, I. A. and Singh, S. R.},
  journal={arXiv preprint arXiv:2411.16745},
  year={2024},
  doi={10.48550/arXiv.2411.16745},
}

@misc{zhouFL,
  author = {Zhou, Shenglong and Li, Geoffrey Ye},
  title  = {Federated Learning via Inexact ADMM},
  year   = {2023},
  url    = {https://arxiv.org/abs/2204.10607}
}

@article{liu2020primer,
  author  = {Liu, Sijia and Chen, Pin-Yu and Kailkhura, Bhavya and Zhang, Gaoyuan and Hero, Alfred and Varshney, Pramod K.},
  title   = {A Primer on Zeroth-Order Optimization in Signal Processing and Machine Learning},
  journal = {IEEE Signal Processing Magazine},
  year    = {2020},
    doi     = {https://doi.org/10.1109/MSP.2020.3003837}
,
  pages={25--43},
  volume={37},
  number={5}
}

@article{chen2023generalizedsmooth,
  title={Generalized-Smooth Nonconvex Optimization is As Efficient As Smooth Nonconvex Optimization},
  author={Chen, Ziyi and Zhou, Yi and Liang, Yingbin and Lu, Zhaosong},
  journal={arXiv preprint arXiv:2303.02854},
  year={2023},
  eprint={2303.02854},
  archivePrefix={arXiv},
  primaryClass={math.OC},
  url={https://arxiv.org/abs/2303.02854}
}

@article{zhang2020gradient,
    title={Why gradient clipping accelerates training: A theoretical justification for adaptivity},
    author={Zhang, Jingzhao and He, Tianxing and Sra, Suvrit and Jadbabaie, Ali},
    journal={arXiv preprint arXiv:1905.11881},
    year={2020}
}

@article{vyguzov2025frankwolfe,
    title={Frank-Wolfe Algorithms for (L0, L1)-smooth functions},
    author={Vyguzov, A.A. and Stonyakin, F.S.},
    journal={arXiv preprint},
    volume={arXiv:2510.16468},
    year={2025},
    url={https://doi.org/10.48550/arXiv.2510.16468},
    archivePrefix={arXiv},
    eprint={2510.16468},
    primaryClass={math.OC}
}

@article{koloskova2023revisiting,
    title={Revisiting Gradient Clipping: Stochastic bias and tight convergence guarantees},
    author={Koloskova, Anastasia and Hendrikx, Hadrien and Stich, Sebastian U.},
    journal={arXiv preprint},
    volume={arXiv:2305.01588},
    year={2023},
    url={https://doi.org/10.48550/arXiv.2305.01588},
    archivePrefix={arXiv},
    eprint={2305.01588},
    primaryClass={cs.LG}
}

@inproceedings{zhang2020improved,
    title={Improved Analysis of Clipping Algorithms for Non-convex Optimization},
    author={Zhang, Bohang and Jin, Jikai and Fang, Cong and Wang, Liwei},
    booktitle={Advances in Neural Information Processing Systems (NeurIPS)},
    year={2020},
    volume={33},
    pages={15511--15521},
    url={https://doi.org/10.48550/arXiv.2010.02519},
    archivePrefix={arXiv},
    eprint={2010.02519},
}

@InProceedings{pmlr-v89-vaswani19a,
  title = 	 {Fast and Faster Convergence of SGD for Over-Parameterized Models and an Accelerated Perceptron},
  author =       {Vaswani, Sharan and Bach, Francis and Schmidt, Mark},
  booktitle = 	 {Proceedings of the Twenty-Second International Conference on Artificial Intelligence and Statistics},
  pages = 	 {1195--1204},
  year = 	 {2019},
  editor = 	 {Chaudhuri, Kamalika and Sugiyama, Masashi},
  volume = 	 {89},
  series = 	 {Proceedings of Machine Learning Research},
  month = 	 {16--18 Apr},
  publisher =    {PMLR},
  url = 	 {https://proceedings.mlr.press/v89/vaswani19a.html},
  
}

@article{gao2024momentum_polyak,
    title={Non-convex Stochastic Composite Optimization with Polyak Momentum},
    author={Gao, Yuan and Rodomanov, Anton and Stich, Sebastian U.},
    journal={arXiv preprint arXiv:2403.02967},
    year={2024},
    eprint={2403.02967},
    archivePrefix={arXiv},
    primaryClass={math.OC},
    url={https://arxiv.org/abs/2403.02967}
}

@article{takezawa2024parameterfree,
    title={Parameter-free Clipped Gradient Descent Meets Polyak},
    author={Takezawa, Yuki and Bao, Han and Sato, Ryoma and Niwa, Kenta and Yamada, Makoto},
    journal={Advances in Neural Information Processing Systems (NeurIPS)},
    year={2024},
    eprint={2405.15010},
    archivePrefix={arXiv},
    primaryClass={cs.LG},
    url={https://arxiv.org/abs/2405.15010}
}

@article{loizou2020SGD_polyak,
    title={Stochastic Polyak Step-size for SGD: An Adaptive Learning Rate for Fast Convergence},
    author={Loizou, Nicolas and Vaswani, Sharan and Laradji, Issam and Lacoste-Julien, Simon},
    journal={Proceedings of the 24th International Conference on Artificial Intelligence and Statistics (AISTATS)},
    year={2021},
    eprint={2002.10542},
    archivePrefix={arXiv},
    primaryClass={math.OC},
    url={https://arxiv.org/abs/2002.10542}
}

@article{hazan2019polyak,
    title={Revisiting the Polyak step size},
    author={Hazan, Elad and Kakade, Sham},
    journal={arXiv preprint arXiv:1905.00313},
    year={2019},
    eprint={1905.00313},
    archivePrefix={arXiv},
    primaryClass={math.OC},
    url={https://arxiv.org/abs/1905.00313}
}

@article{vankov2024optimizingl0l1,
  title={Optimizing $(L_0, L_1)$-Smooth Functions by Gradient Methods},
  author={Vankov, Daniil and Rodomanov, Anton and Nedich, Angelia and Sankar, Lalitha and Stich, Sebastian U.},
  journal={arXiv preprint arXiv:2410.10800},
  year={2024},
  eprint={2410.10800},
  archivePrefix={arXiv},
  primaryClass={math.OC},
  url={https://arxiv.org/abs/2410.10800}
}

@article{gorbunov2024methodsl0l1,
  title={Methods for Convex $(L_0,L_1)$-Smooth Optimization: Clipping, Acceleration, and Adaptivity},
  author={Gorbunov, Eduard and Tupitsa, Nazarii and Choudhury, Sayantan and Aliev, Alen and Richt{\'a}rik, Peter and Horv{\'a}th, Samuel and Tak{\'a}{\v c}, Martin},
  journal={arXiv preprint arXiv:2409.14989},
  year={2024},
  eprint={2409.14989},
  archivePrefix={arXiv},
  primaryClass={math.OC},
  url={https://arxiv.org/abs/2409.14989}
}

@misc{comparisons,
  author = {Zhang, Chenyi and Li, Tongyang},
  title  = {Comparisons Are All You Need for Optimizing Smooth Functions},
  year   = {2024},
  url    = {https://arxiv.org/abs/2405.11454}
}
